%% file: main.tex
\RequirePackage[2020-02-02]{latexrelease}
\documentclass[12pt]{amsart}

\usepackage[T1]{fontenc}
\usepackage{subcaption}
\usepackage{hyperref}
\usepackage[numbers]{natbib}
\usepackage{indentfirst}
\usepackage{amsmath}
\usepackage{amsthm}
\usepackage[printwatermark]{xwatermark}
\usepackage{xcolor}
\usepackage{graphicx}
\usepackage{lipsum}
\usepackage{amssymb}
\usepackage{graphicx}
\usepackage{geometry}
\usepackage{float}
\newtheorem{theorem}{Theorem}[section]
\newtheorem{corollary}{Corollary}

\newtheorem{lemma}[theorem]{Lemma}
\newtheorem{prop}[theorem]{Proposition}

\usepackage{setspace}
\input{preamble}

\begin{document}

\setlength{\abovedisplayskip}{0pt}
\setlength{\belowdisplayskip}{0pt}
\setlength{\abovedisplayshortskip}{0pt}
\setlength{\belowdisplayshortskip}{0pt}
\vspace{1.5in}
\include{cover}

\maketitle

\vspace{0.1in}

\begin{centering}
Under the direction of

Zhenhao Li

Graduate Student

Massachusetts Institute of Technology

\end{centering}
\vspace{0.5in}

\begin{abstract}
\input{abstract}
\end{abstract}
\nopagebreak

\include{paper}

\bibliographystyle{style}

\bibliography{biblio}

\end{document}

%% file: preamble.tex
\usepackage{color,colortbl}

\theoremstyle{definition}
\newtheorem{definition}{Definition}[section]
\usepackage[capitalize]{cleveref}
	\crefname{claim}{Claim}{Claims}
	\Crefname{claim}{Claim}{Claims}
	\crefname{app-corollary}{Corollary}{Corollaries}
	\Crefname{app-corollary}{Corollary}{Corollaries}
	\crefname{app-definition}{Definition}{Definitions}
	\Crefname{app-definition}{Definition}{Definitions}
	\crefname{figure}{Figure}{Figures}
	\Crefname{figure}{Figure}{Figures}
	\crefname{lemma}{Lemma}{Lemmata}
	\Crefname{lemma}{Lemma}{Lemmata}
	\crefname{app-lemma}{Lemma}{Lemmata}
	\Crefname{app-lemma}{Lemma}{Lemmata}
	\crefname{app-proposition}{Proposition}{Proposition}
	\Crefname{app-proposition}{Proposition}{Proposition}
	\crefname{app-theorem}{Theorem}{Theorems}
	\Crefname{app-theorem}{Theorem}{Theorems}

%% file: cover.tex


\title[Smoothness and Regularity in the Poincaré Problem]{On the Smoothness and Regularity of the Chess Billiard Flow and the Poincaré Problem}

\author{Sally Zhu}


%% file: abstract.tex

\noindent The Poincaré problem is a model of two-dimensional internal waves in stable-stratified fluid. The chess billiard flow, a variation of a typical billiard flow, drives the formation behind and describes the evolution of these internal waves, and its trajectories can be represented as rotations around the boundary of a given domain. We find that for sufficiently irrational rotation in the square, or when the rotation number $r(\lambda)$ is Diophantine, the regularity of the solution $u(t)$ of the evolution problem correlates directly to the regularity of the forcing function $f(x)$. Additionally, we show that when $f$ is smooth, then $u$ is also smooth. These results extend studies that have examined singularity points, or the lack of regularity, in rational rotations of the chess billiard flow. We also present numerical simulations in various geometries that analyze plateau formation and fractal dimension in $r(\lambda)$ and conjecture an extension of our results. Our results can be applied in modeling two dimensional oceanic waves, and they also relate the classical quantum correspondence to fluid study. 

%% file: paper.tex
\newtheorem{conjecture}[theorem]{Conjecture}
\newtheorem{thm}[theorem]{Theorem}

\section{Introduction}

Internal waves are important to the study of oceanography and to the theory of rotating fluids. The waves describe how an originally unmoving fluid can move and evolve under a periodic forcing function. We study the behavior of a particular two-dimensional model for these internal waves, called the Poincaré problem, which forms patterns called billiard flows.

Billiard flows are a type of mapping that maps points on a boundary of a given shape to another point on that boundary, with some sort of reflection or bouncing step. For example, the most familiar billiard flow is the pool billiard mapping in a game of pool, in which the ball is bounced off one side of the table and follows a reflective trajectory in which the angle of incidence equals the angle of exit. 

We consider a variation of that billiard flow, called the \textit{chess billiard flow}. Rather than preserving the angles of reflection, this map instead preserves the slopes of the trajectories. Each mapping $b$ consists of traveling first on a line of slope $\rho$, and then bouncing off the boundary at a slope of $-\rho$. We visualize this in Figure \ref{fig:chessbilliard} from Dyatlov et al. \cite{DWZ}, in which we start at $x$, travel on the blue line of some slope $\rho$, then bounce off the side and traveling on the red line of slope $-\rho$, to $b(x)$. 

\begin{figure}[H]
    \centering
    \includegraphics[width=6cm]{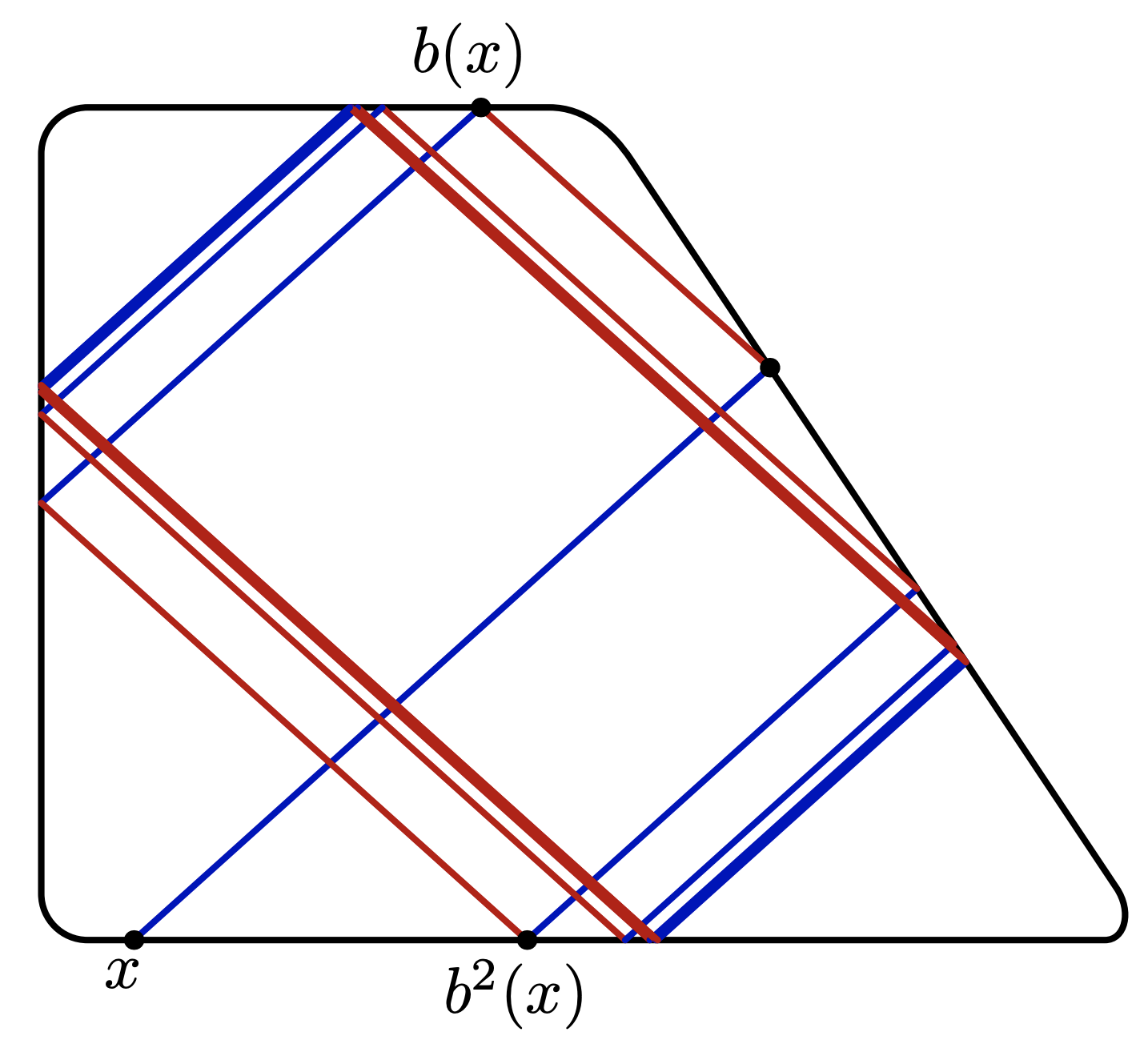}
    \caption{Depiction of a series of chess billiard mappings on a trapezoid, traveling across the parallel red and blue lines, from $x$ to $b(x)$ to $b^2(x)$ and so on, from Dyatlov et al. \cite{DWZ}}
    \label{fig:chessbilliard}
\end{figure}

The figure shows the mapping from $x$ to $b(x)$ to $b^2(x)$, and so on, recursively traveling on the blue and red lines. Note that each mapping consists of two movements: one blue line, then one red line. 

Another way to describe the chess billiard map is as a rotation of points along the boundary. Figure \ref{fig:squarebilliardrotationleft} shows one mapping on the square, from $x$ to $b(x)$. We can imagine $x$ as being rotated counterclockwise to $b(x)$, following the arrow in the figure. Figure \ref{fig:squarebilliardrotationright} further depicts this, where we can visualize $b(x)$ being rotated to $b^2(x)$ following the arrow. This rotation allows us to more easily understand how the chess billiard map maps points. 

\begin{figure}[H]
    \centering
    \begin{subfigure}[b]{0.3\textwidth}
        \centering
         \includegraphics[width=\textwidth]{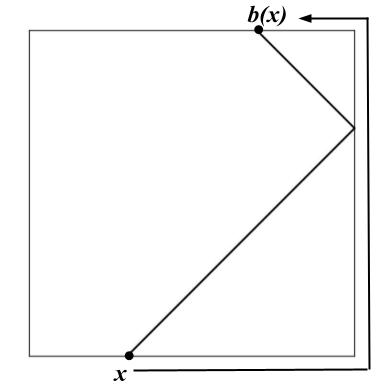}
         \caption{Mapping from $x$ to $b(x)$.}
         \label{fig:squarebilliardrotationleft}
    \end{subfigure}
    \hspace{1cm}
    \begin{subfigure}[b]{0.3\textwidth}
        \centering
         \includegraphics[width=\textwidth]{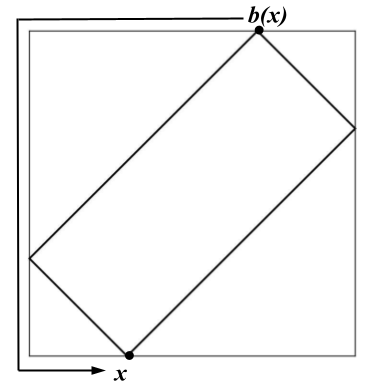}
         \caption{Mapping from $b(x)$ to $x=b^2(x)$.}
         \label{fig:squarebilliardrotationright}
    \end{subfigure}
    \caption{A depiction of how the chess billiard map can be viewed as a rotation of points along the boundary.}
    \label{fig:squarebilliardrotation}
\end{figure}

We quantify this rotation using the \textit{rotation number} $r$, which can be thought of as the average rotation per mapping over time. For example, in Figure \ref{fig:squarebilliardrotationright}, since it takes two mappings of $b$ to complete a full rotation (as $x = b^2(x)$), then, on average, each mapping travels $1/2$ of the boundary. Hence, the rotation number is $r = \frac{1}{2}$. 

Now we discuss rational and irrational rotation, which correspond directly to the rationality or irrationality of $r$. When $r$ is rational, as in Figure \ref{fig:squarebilliardrotation}, we can see a periodic trajectory—the same lines are traveled along again and again. However, in an irrational rotation (i.e. $r$ is irrational), a point can never be mapped to itself again in some integer number of rotations, so we don’t expect any periodic trajectories. The chess billiard flow has been previously studied in depth for rational rotation, and we study irrational rotation. 

We look at the wave problem known as the Poincaré problem (see Equation (\ref{eq:1}) in Section \ref{section2}), which is a two-dimensional model that describes how an originally unmoving stable-stratified fluid can move and evolve. Specifically, the solution $u$ describes the behavior of the waves formed, which follows the chess billiard mappings. The direct connection between the Poincaré problem and the chess billiard flow has been shown and verified mathematically and experimentally \cite{Hazewinkle,Lenci}. 

We examine the differentiability of $u$ for different types of trajectories of the billiard flow. Previously, Dyatlov et al \cite{DWZ} showed that when $r$ is rational (i.e. it is a periodic trajectory), then $u$ is not smooth (i.e. $u$ is not highly differentiable). This is because singularity points (points where the derivative does not exist) form along the trajectory of the billiard flow. For example, in the illustrations by Dyatlov et al. \cite{DWZ} in Figure \ref{fig:irrationalrational}, distinct lines form in the rational rotation in Figure \ref{fig:rirleft}, where $u$ is not smooth. We instead investigate differentiation in cases of irrational rotation, in which case we don’t expect the formation of singularities (since no obvious trajectory exists), and instead expect the waves to eventually smooth out. For example, in the near-irrational rotation in Figure \ref{fig:rirright}, the solution seems much smoother. 

\begin{figure}[H]
    \centering
    \begin{subfigure}[b]{0.4\textwidth}
        \centering
         \includegraphics[width=\textwidth]{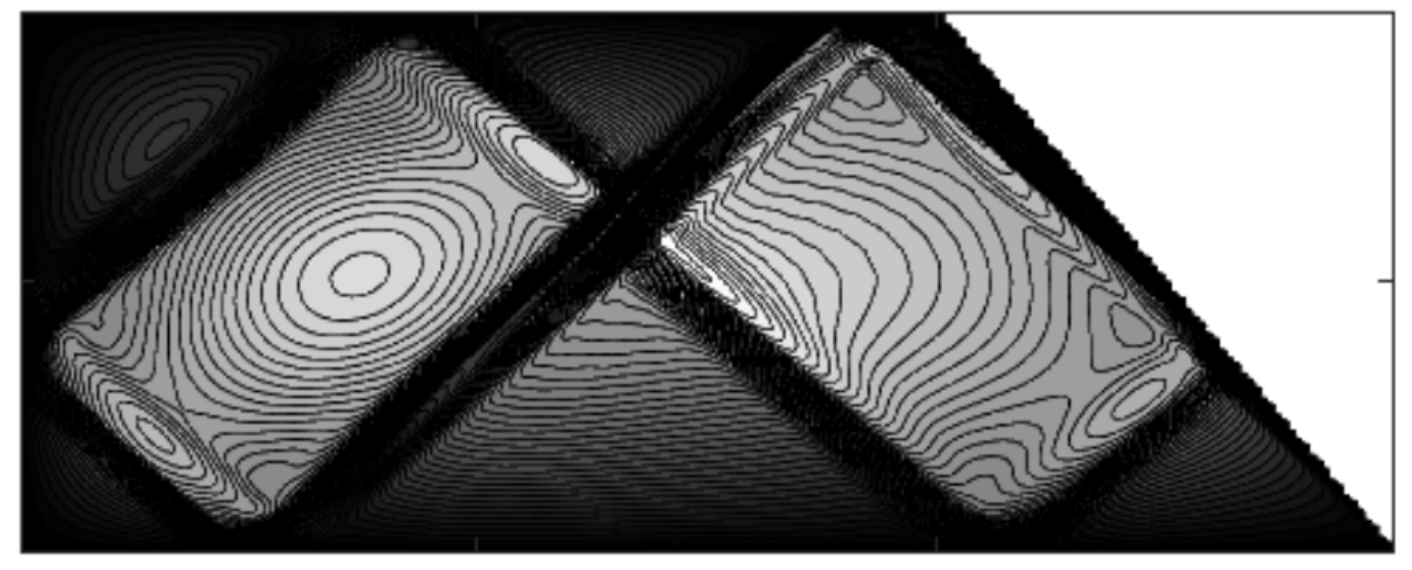}
         \caption{An illustration of rational rotation with singularity points.}
         \label{fig:rirleft}
    \end{subfigure}
    \hspace{1cm}
    \begin{subfigure}[b]{0.4\textwidth}
        \centering
         \includegraphics[width=\textwidth]{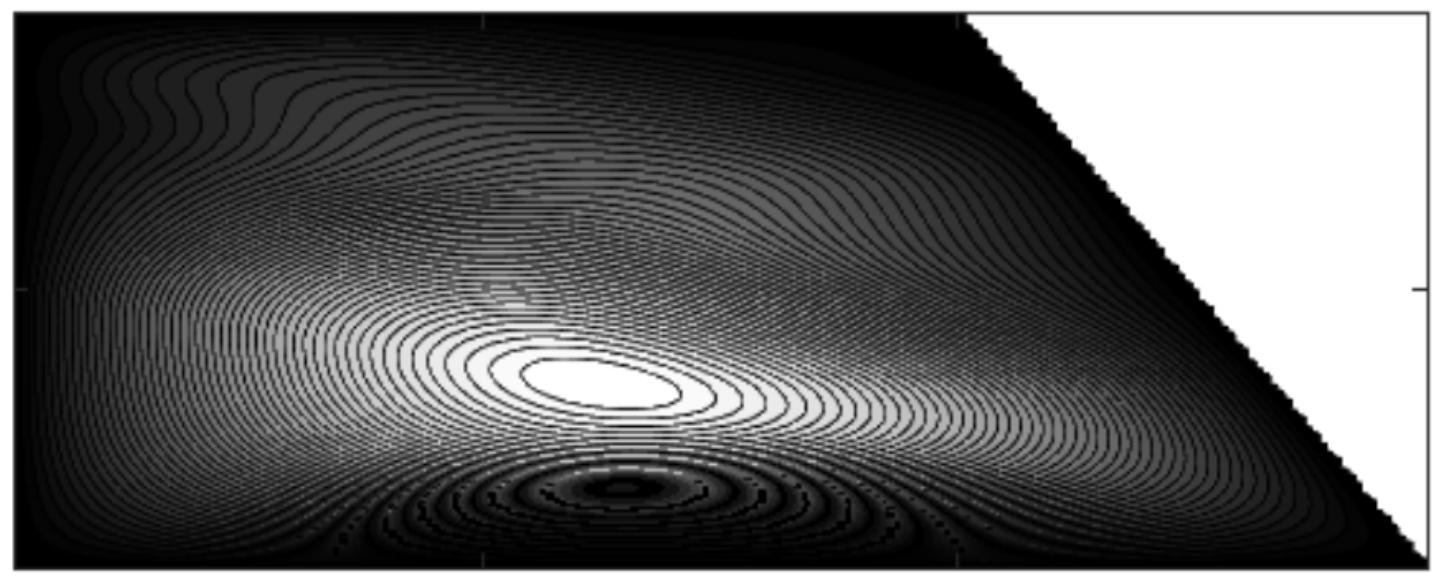}
         \caption{An illustration of nearly irrational rotation, which appears smooth.}
         \label{fig:rirright}
    \end{subfigure}
    \caption{Illustrations of rational vs. irrational rotations, from Dyatlov et al. \cite{DWZ}. Note that we assume much stronger irrationality conditions than the irrational rotation in panel (b).}
    \label{fig:irrationalrational}
\end{figure}

We show that sufficiently irrational rotation results in highly differentiable and smooth solutions $u$. 

This paper is organized as follows. In Section \ref{section2}, we establish the necessary definitions and prove supporting lemmas. In Section \ref{section3}, we present the proof of our main result. Finally, in Section \ref{section4}, we present numerical simulations and conjecture an extension of our result. 

\section{Preliminaries} \label{section2}

\subsection{The Chess Billiard Map and Rotation in the Square}

We more formally describe the chess billiard mapping, as in Dyatlov et al \cite{DWZ}. Given a domain $\Omega$ and its boundary $\partial \Omega$, we write the slopes in terms of $\lambda$, as $\rho = \dfrac{\sqrt{1-\lambda^2}}{\lambda}$ and $-\rho = -\dfrac{\sqrt{1-\lambda^2}}{\lambda}$, following Dyatlov et al \cite{DWZ}. Then, we write the two lines through $(p_1,p_2)$ with those slopes as
$$ y-p_2 = \pm \dfrac{\sqrt{1-\lambda^2}}{\lambda}(x-p_1), $$
for a given point $p = (p_1, p_2) \in \partial \Omega$, and $\lambda \in (0,1)$. In Figure \ref{fig:chessbilliard}, the series of parallel blue lines represent one of the set of lines with slope $\sqrt{1-\lambda^2}/\lambda$, and the red lines represent the other set of parallel lines, with the negative slope $-\sqrt{1-\lambda^2}/\lambda$.  

The first step of the chess billiard map is to take $p$ to the unique other point of intersection between $\partial \Omega$ and the line $y-p_2 = \dfrac{\sqrt{1-\lambda^2}}{\lambda}(x-p_1),$ essentially traveling from $p$ along a blue line. Call this point $p’$. The next step in the map is to take $p’$ to the unique other point of intersection between $\partial \Omega$ and the line $y-p’_2 = -\dfrac{\sqrt{1-\lambda^2}}{\lambda}(x-p’_1),$ essentially traveling from $p’$ along a red line. Then, we end up at $b(p,\lambda)$, completing one mapping. 

Next, we more formally describe the rotation number and how to compute it. To calculate the fraction rotated around the boundary, we can measure how much distance is traveled along the boundary for a given rotation. For example, in Figure 2, we see both of the mappings travel $1/2$ of the distance of the boundary; hence, the rotation number $r = 1/2$. We keep summing the fractional distance around $\partial \Omega$ for a total of $k$ mappings, and divide the sum by $k$ to get an average. The rotation number is this average as $k$ grows very large, and is written as 
\begin{equation} \label{eq:rotlimit}
r(p, \lambda) = \displaystyle \lim_{k_t \rightarrow \infty} \dfrac{d_t}{k_t}
\end{equation}
where $d_t$ is the distance traveled in $k_t$ iterations of the map, $p$ is the original starting point on $\partial \Omega$, and $\lambda$ is the term given to $b(p,\lambda)$. 

We have the following proposition from Brin and Stuck \cite{dynamical} about the rotation number. 

\begin{prop}
The limit used to define the rotation number (Equation (\ref{eq:rotlimit})) exists and is independent of the starting point $p$ and distance parameterization. 
\end{prop}

Hence, we can write the rotation number only in terms of $\lambda$, as $r(\lambda)$. From here, we find an explicit form of $r(\lambda)$ of the chess billiard map in the square. 

\begin{prop} \label{prop:rotexplicit}
The rotation number $r(\lambda)$ of the chess billiard flow in a square is given by $\dfrac{\lambda}{\sqrt{1-\lambda^2}+\lambda}$.
\end{prop}

\begin{proof}

We call the square that $b$ acts on as $\Omega$. For the sake of simplicity, we let the side length of $\Omega$ be $1/4$. This way, the perimeter is 1, so the distance traveled is equivalent to the fractional distance traveled. 

To measure the distance traveled, we reflect $\Omega$ repeatedly over its right edge, as shown by Figure \ref{fig:reflections}. We see that this extends all of the lines in the chess billiard map. By transforming $\Omega$ to an infinite domain, which we call $S$, where points do not intersect with themselves, it is much simpler to measure distance (essentially distinguishing between displacement and distance). 

\begin{figure}[H]
    \centering
    \includegraphics[width=17cm]{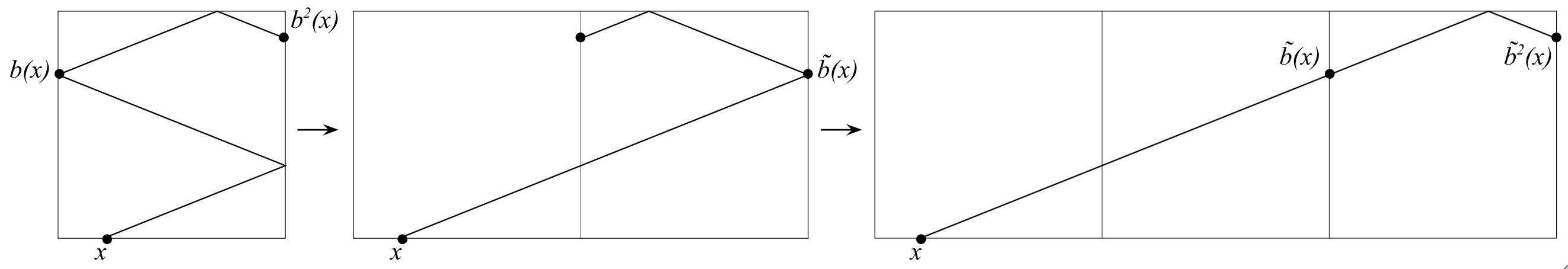}
    \caption{The reflection of a square with the chess billiard map repeatedly across its right-most edge, which extends the lines in the original square.}
    \label{fig:reflections}
\end{figure}

After these reflections, we see that lines switch direction only by hitting either the top or bottom boundary of $S$. We measure the distance traveled in each mapping by traveling along each $\frac{1}{4} \times \frac{1}{4}$ square in $S$, as shown by the highlighted red path in Figure \ref{fig:trace}. Figure \ref{fig:trace} shows how the distances we trace for one mapping of $b$ in $\Omega$ and in the reflections of $S$ are equal. Instead of wrapping around itself in $\Omega$, in $S$, the distance is measured by moving from bottom edge to top edge after each square. 

\begin{figure}[H]
    \centering
    \includegraphics[width=11cm]{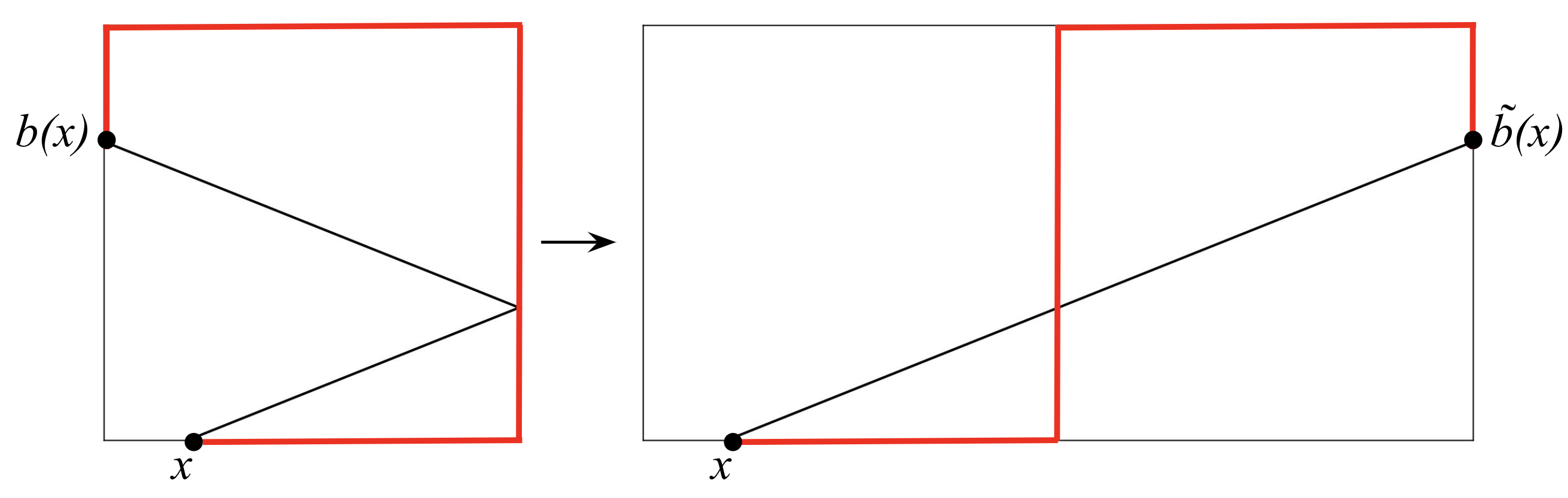}
    \caption{The total distances traveled in one mapping of $b(\lambda)$, highlighted in red, are equal in $\Omega$ and in $S$.}
    \label{fig:trace}
\end{figure}

Hence, the total distance can be found by summing the horizontal and vertical distance. We choose the starting point to be the lower-left corner (as the starting point is independent of $r(\lambda)$). Starting from that corner, let the number of times the trajectory reaches the upper or lower boundary (i.e. how many times the line switches direction from up to down) be $k$. To simplify calculations, because the slope of the lines are $\pm\frac{\sqrt{1-\lambda^2}}{\lambda}$ (given in the definition of the chess billiard map), we write the slopes as $\pm \rho = \pm\frac{\sqrt{1-\lambda^2}}{\lambda}$ for now. 

Because the magnitude of the lines is $\rho$, the $k$ switches move horizontally across a total of $\left \lfloor\dfrac{k}{\rho}\right \rfloor$ of the $\frac{1}{4} \times\frac{1}{4}$ squares. The squares have side length $\frac{1}{4}$, so the horizontal distance traveled is $\dfrac{k}{4\rho}$, and the vertical distance traveled is $\dfrac{1}{4}\left \lfloor \dfrac{k}{\rho} \right \rfloor$. 

We count the total number of mappings of $b(\lambda)$ applied by counting how many times the lines hit a side of any of the $\frac{1}{4}\times\frac{1}{4}$ squares (each mapping $b(\lambda)$ consists of two of these bounces). We observe that the top and bottom edges are hit a total of $k$ times. Using the same reasoning used to count how many squares were crossed horizontally, we find that the number of left or right edges crossed is $\left\lfloor \dfrac{k}{\rho} \right\rfloor$. Each of the mappings $b$ require two bounces, so the total number of mappings is 
\begin{center}
    $\dfrac{\left\lfloor \dfrac{k}{\rho} \right\rfloor + k}{2}$.
\end{center}
To calculate the rotation number, we take the limit
$r(\lambda) = \displaystyle \lim_{k_t \rightarrow \infty} \dfrac{d_t}{k_t}$, where $d_t$ is the distance traveled in $k_t$ iterations. Hence, this limit is
\begin{equation*}
    r(\lambda) = \displaystyle \lim_{k \rightarrow \infty} \dfrac{\dfrac{k}{4\rho} + \dfrac{1}{4}\left \lfloor \dfrac{k}{\rho} \right \rfloor}{\dfrac{\left\lfloor \dfrac{k}{\rho} \right\rfloor + k}{2}}
    = \displaystyle \lim_{k \rightarrow \infty} \dfrac{\dfrac{k}{2\rho}+\dfrac{1}{2}\left \lfloor \dfrac{k}{\rho} \right \rfloor}{\left\lfloor \dfrac{k}{\rho} \right\rfloor+k} = \lim_{k \rightarrow \infty} \dfrac{\dfrac{k}{2\rho} + \dfrac{1}{2}\cdot \dfrac{k}{\rho}}{\dfrac{k}{\rho} + k} = \lim_{k \rightarrow \infty} \dfrac{k}{k+k\rho} = \dfrac{1}{1+\rho}.
\end{equation*}
And because $\rho = \frac{\sqrt{1-\lambda^2}}{\lambda}$, we can substitute this in to conclude that
$$r(\lambda) = \dfrac{\lambda}{\sqrt{1-\lambda^2}+\lambda}.$$
\end{proof} 

\subsection{Irrational Rotation of the Map}

Here, we define what it means for the chess billiard map to have sufficiently irrational rotation. Recall that the rotation is irrational when $r(\lambda)$ is irrational. We use the definition of a \textit{Diophantine irrational} as our measure of sufficiently irrational rotation. 

\begin{definition} \label{def:diophantine}
Let $r$ be an irrational number. We call $r$ $\beta-$\textit{Diophantine} if for all rationals $p/q$, with $q \in \mathbb{Z}^+$, there exist some constants $\beta$ and $C >0$, for which $r$ satisfies the inequality
\begin{equation} \label{eq:diophantineinequality}
    \displaystyle \left\lvert r-\dfrac{p}{q} \right\rvert \ge \dfrac{C}{q^{2+\beta}}.
\end{equation}
\end{definition}

Loosely, $r$ being Diophantine means that it is far from all rational numbers. The set of $\beta-$Diophantine irrationals for a given $\beta$ is also positively dense (i.e. has full measure) \cite{hubbard}, which loosely means that most numbers are Diophantine. Alternatively speaking, if a number in $\mathbb{R}$ is chosen at random, it is a Diophantine irrational; so, this definition covers essentially all numbers. 

\subsection{The Wave Equation}

We study the chess billiard map and rotation in the context of the Poincaré problem, which is an equation modeling waves and wave evolution. It is given by 
\begin{equation} \label{eq:1}
     (\partial_t^2 \Delta + \partial_{x_2}^2)u = f(x) \cos{\lambda t}, \quad \quad u|_{t=0} = \partial_t u|_{t=0}=0, \quad \quad u|_{\partial\Omega} = 0,
\end{equation}
where $\Delta = \partial_{x_1}^2 + \partial_{x_2}^2$, $\Omega$ is a smooth convex domain in $\mathbb{R}^2$ and $\partial\Omega$ is its boundary, $f(x)$ is some forcing function in $C(\Omega)$, and $\lambda \in (0,1)$ is the frequency of the periodic forcing given by $\cos{(\lambda t)}$. We call ``$u
$" the solution to the Poincaré problem, and it describes the behavior of the waves (more specifically, through describing fluid velocity). 

This equation is directly correlated with the chess billiard flow. We can establish this relationship by observing resolvents of the differential operator of the equation near the resonant frequency of $\lambda$. This has also been confirmed with experiments that show the waves evolve onto linear paths and trajectories and form the chess billiard flow \cite{Hazewinkle}, where the slopes in $b$ (i.e., $\frac{\lambda}{\sqrt{1-\lambda^2}}$) are determined by the $\lambda$ in the $\cos{(\lambda t)}$ forcing term \cite{maasetal}. The function $u$ is what we represent by the chess billiard flow in the fluid, and it is what we show is highly differentiable and smooth. 

In order to study this smoothness, we use the geometric properties of the chess billiard flow; most importantly, the rotation number. By looking at irrational rotation (i.e. irrational $r(\lambda)$) in the chess billiard flow, we can better understand when $u$ will be smooth or unsmooth. 

\subsection{Differentiability and Smoothness}

We say that a function $u$ is smooth when it is infinitely differentiable. To describe a function's differentiability (i.e. regularity), we use the notation $C^q,$ which means that the function is differentiable $q$ times. A smooth function is in $C^\infty$. From here, we will define the Sobolev spaces, which gives us a metric that allows us to quantify what the maximum $q$ in $C^q$ is for a given function. 

\begin{definition}
Let $\Omega=[0,1]\times[0,1]$. We define $L^2(\Omega)$ as 
\begin{equation*}
    L^2(\Omega)=\{f \text{ measurable}: \int \lvert f \rvert^2 < \infty\}.
\end{equation*}
\end{definition}

From here, we define the Sobolev spaces $H^s$. 

\begin{definition} \label{def:sobolev}
The Sobolev space $H^s(\Omega)$ is defined as
\begin{equation*}
    H^s(\Omega) = \{ f \in L^2: \displaystyle \sum_{k=-\infty}^\infty (1+k_1^2+k_2^2)^s \lvert \hat{f}(k_1,k)2) \rvert^2 < \infty\},
\end{equation*}
where $\hat{f}(k_1,k_2)$ are the Fourier coefficients of $f$ (see Section blank). 
\end{definition}

To connect the Sobolev spaces to a function's regularity, we state the following proposition from Friedlander et al. \cite{friedlander}. 

\begin{prop} \label{prop:htoc}
If a function $g$ is in $H^s$, then it is also in $C^{s-1}$. 
\end{prop}

Finally, given this set up, we proceed to the proof of our main result. 

\section{Proof of Main Result} \label{section3}

The main result we prove in this section regards the smoothness of the solution $u$ of the Poincaré problem in the square for different values of $\lambda$. We show that the regularity (i.e. differentiability) of $u$ relates directly to the regularity of the forcing function $f$, and that $u$ is highly regular or smooth for sufficiently irrational $r(\lambda)$ and a fixed $f$. More formally, we show the following theorem. 

\begin{theorem}\label{theorem:bigtheorem}
Given a forcing function $f(x) \in C^s [0,1]\times[0,1]$ and a $\beta$-Diophantine rotation number $r(\lambda)$ for some $\beta$ and $C > 0$, the solution $u(t)$ of the Poincaré problem in the square is in $C^{s-1-\beta}$. 
\end{theorem}

We build up to the proof in the next subsections. In Subsection \ref{solvefourier}, we establish the norms for the Fourier series, and explicitly calculate the Fourier coefficients of $u$. In Subsection \ref{boundfourier}, we establish the irrational condition on $r(\lambda)$ and find an upper bound for the Fourier coefficients of $u$. Finally, in Subsection \ref{regularity}, we use the Sobolev space definition and propositions to relate the smoothness of $f$ to $u$. 

\subsection{Solving for the Fourier coefficients of $u$} \label{solvefourier}

Recall that the Fourier coefficients of a periodic function $f(x)$ where $x = (x_1, x_2) \in [0,1] \times [0,1]$ are
\begin{equation*}
    \hat{f}(k_1,k_2) = \displaystyle \int_0^1 \int_0^1 f(x_1,x_2)e^{-2\pi i(x_1k_1+x_2k_2)}dx_1 dx_2, 
\end{equation*}
for $k_1, k_2 \in \mathbf{Z}$.

We use these coefficients to write the Fourier Series of $f$, which decomposes $f$ into periodic functions of smaller amplitudes, and is useful in differentiation and determining differentiability. We decompose $f$ as
\begin{equation*}
    f(x_1,x_2) = \displaystyle \sum_{k_1=-\infty}^{\infty} \sum_{k_2=-\infty}^{\infty}  \hat{f}(k_1,k_2) e^{2\pi i(x_1k_1+x_2k_2)},
\end{equation*}
where $\hat{f}$ are the Fourier coefficients.

To compute the Fourier coefficients of $u$, we first formally take a Fourier transform of both sides of the Poincaré problem, restated below: 
$$(\partial_t^2 \Delta + \partial_{x_2}^2)u = f(x) \cos{\lambda t}, \quad \quad u|_{t=0} = \partial_t u|_{t=0}=0, \quad \quad u|_{\partial\Omega} = 0,$$
where $\Delta=\partial_{x_1}^2 + \partial_{x_2}^2$. 

We use the following proposition, which follows from the one variable case presented in Beals \cite{analysistxtbk}, to take the Fourier transform of the derivative of a function. 

\begin{prop} \label{prop:fourierderivative}
For a continuous periodic function $f$ on $(x_1,x_2) \in [0,1] \times [0,1]$ with a continuous derivative $f'$, we have
\begin{center}
    $\widehat{\dfrac{\partial}{\partial x_1}f} = -2\pi i k_1 \hat{f}(x_1,x_2)$ and
    $\widehat{\dfrac{\partial}{\partial x_2}f} = -2\pi i k_2 \hat{f}(x_1,x_2)$.
\end{center}
\end{prop}

We use this on Equation (\ref{eq:1}). Firstly, using Proposition \ref{prop:fourierderivative}, it follows that $\widehat{\dfrac{\partial}{\partial x_1}u} = 2 \pi i k_1 \hat{u}$ and $\widehat{\dfrac{\partial}{\partial x_2}u} = 2 \pi i k_2 \hat{u}$. Hence, $\widehat{\partial_{x_1}^2 u} = -4 \pi^2 k_1^2 \hat{u}$ and $\widehat{\partial_{x_2}^2 u} = -4 \pi^2 k_2^2 \hat{u}$.

Taking the Fourier transform of both sides of Equation (\ref{eq:1}) in $x$, the right hand side simply is $\hat{f}(x) \cos \lambda t$ (because $\cos \lambda t$ is not a function of $x$). 

On the left hand side, we have 
$$(\partial_t^2(k_1^2+k_2^2)+k_2^2) \cdot (-4\pi^2) \hat{u},$$
and by factoring out all instances of $-4 \pi^2$,  we conclude that $u$ is a solution to
\begin{equation} \label{eq:3}
    -4\pi^2 (\partial_t^2 (k_1^2+k_2^2) + k_2^2) \hat{u} = \hat{f}(k_1,k_2)\cos \lambda t.
\end{equation}

Now, Equation (\ref{eq:3}) can be solved as an ordinary differential equation, which we do with Wolfram Alpha. Given the initial conditions presented in Equation (\ref{eq:1}), we solve $\hat{u}(t,k_1,k_2)$ explicitly as a function of $t$ for each $k_1, k_2 \in \mathbb{Z}_{\neq 0}$, as 
\begin{equation} \label{eq:fourierexplicit}
    \hat{u}(t,k_1,k_2) = \dfrac{\hat{f}(k_1,k_2) \cos \lambda t}{4 \pi^2 (-k_2^2 + k_1^2 \lambda^2 + k_2^2 \lambda^2)}.
\end{equation}

Now that the Fourier coefficients are solved, we explicitly find an upper bound for them.

\subsection{Bounding the Fourier coefficients} \label{boundfourier}

Now, we bound the Fourier coefficients of $u$. We wish to bound them because we want to show that the coefficients $\hat{u}(k)$ decay rapidly for large $k$, which suggests that $u$ would be smooth. 

In this subsection, we prove the following proposition bounding $\hat{u}$ in terms of $\hat{f}$ and $\lambda$. 

\begin{prop}
When $r(\lambda)$ is $\beta$-Diophantine for some $C$ and $\beta > 0$, we have
\begin{equation*} \label{prop:bigboundingprop}
    \lvert \hat{u}(t,k) \rvert < \vert \hat{f}(x)\rvert \cdot C \cdot \left(\sqrt{1+k_1^2+k_2^2}\right)^{1+\beta}.
\end{equation*}
\end{prop}

To establish this upper bound, we first prove an auxiliary lemma. 

\begin{lemma} \label{lemma:biglemma}
If $k_1,k_2 \in \mathbb{Z}_{\neq 0}$ and $r(\lambda)$ is $\beta$-Diophantine for some $\beta, C > 0$, then
\begin{center}
$\displaystyle \left\lvert \dfrac{1}{-k_2^2+k_1^2\lambda^2+k_2^2\lambda^2} \right\rvert < C\left(\sqrt{1+k_1^2+k_2^2}\right)^{1+\beta}$.
\end{center}
\end{lemma}

\begin{proof}

First, we write
\begin{center}
$\displaystyle \left\lvert \dfrac{1}{-k_2^2+k_1^2\lambda^2+k_2^2\lambda^2} \right\rvert =\displaystyle \left\lvert \dfrac{1}{k_1^2\lambda^2-k_2^2(1-\lambda)^2} \right\rvert = \displaystyle \left\lvert\dfrac{1}{\lvert k_1\rvert\lambda - \lvert k_2\rvert\sqrt{1-\lambda^2}} \right\rvert \cdot \displaystyle \left\lvert\dfrac{1}{\lvert k_1\rvert\lambda +\lvert k_2\rvert\sqrt{1-\lambda^2}}\right\rvert$. 
\end{center}

We substitute $r(\lambda)=\dfrac{\lambda}{\sqrt{1-\lambda^2}+\lambda}$ and $p/q= k_2/(k_1+k_2)$ into Equation (\ref{eq:diophantineinequality}), the equation defining $\beta-$Diophantine irrationals, and we have
\begin{equation} \label{eq:6}
    \displaystyle \left\lvert \dfrac{\lambda}{\sqrt{1-\lambda^2}+\lambda} - \dfrac{\lvert k_2\rvert}{\lvert k_1\rvert+\lvert k_2\rvert} \right\rvert > C\left\lvert\dfrac{1}{\lvert k_1\rvert+\lvert k_2\rvert}\right\rvert^{2+\beta} = C \left (\dfrac{1}{\lvert k_1\rvert+\lvert k_2\rvert} \right )^{2+\beta},
\end{equation}
where $C$ is independent of $k_1$ and $k_2$. Note that $C$ may vary throughout the proof, but always remains independent of $k_1$ and $k_2$. 

Putting the fractions under a common denominator, we can write
\begin{align*}
\displaystyle \left\lvert \dfrac{\lambda}{\sqrt{1-\lambda^2}+\lambda} - \dfrac{\lvert k_2\rvert}{\lvert k_1\rvert+\lvert k_2\rvert} \right\rvert &= \displaystyle \left\lvert \dfrac{\lvert k_1\rvert\lambda + \lvert k_2\rvert\lambda - \lvert k_2\rvert\sqrt{1-\lambda^2}-\lvert k_2\rvert\lambda}{(\lvert k_1\rvert+\lvert k_2\rvert)(\sqrt{1-\lambda^2}+\lambda)} \right\rvert \\ &= \displaystyle \left\lvert \dfrac{\lvert k_1\rvert\lambda - \lvert k_2\rvert\sqrt{1-\lambda^2}}{(\lvert k_1\rvert+\lvert k_2\rvert)(\sqrt{1-\lambda^2}+\lambda)} \right\rvert > C\left(\dfrac{1}{\lvert k_1\rvert+\lvert k_2\rvert}\right)^{2+\beta}.
\end{align*}
Simplifying, we have
\begin{center}
    $\displaystyle \left\lvert \dfrac{\lvert k_1\rvert\lambda - \lvert k_2\rvert\sqrt{1-\lambda^2}}{\sqrt{1-\lambda^2}+\lambda} \right\rvert > C\left(\dfrac{1}{\lvert k_1\rvert+\lvert k_2\rvert}\right)^{1+\beta}$. 
\end{center}
Therefore,
\begin{center}
    $\displaystyle \left\lvert \dfrac{\sqrt{1-\lambda^2}+\lambda}{\lvert k_1\rvert\lambda - \lvert k_2\rvert\sqrt{1-\lambda^2}} \right\rvert < \dfrac{1}{C}(\lvert  k_1\rvert+\lvert k_2\rvert) ^{1+\beta}$.
\end{center}
Furthermore, because $0 \le \lambda \le 1$, and $(\lambda + \sqrt{1-\lambda^2})^2 = \lambda^2 + 1-\lambda^2 + 2\lambda\sqrt{1-\lambda^2} = 1+2\lambda\sqrt{1-\lambda^2} \ge 1$, then we know that $\lambda + \sqrt{1-\lambda^2} \ge 1$. 
Hence, we conclude that 
\begin{center}
    $\displaystyle \left\lvert \dfrac{1}{\lvert k_1\rvert\lambda -\lvert k_2\rvert\sqrt{1-\lambda^2}} \right\rvert < \dfrac{1}{C}( \lvert k_1\rvert+\lvert k_2\rvert)^{1+\beta}$.
\end{center}

Next, because $\lvert k_1 \rvert$ and $\lvert k_2 \rvert$ are integers $\ge 1$, we know
\begin{center}
    $\lvert k_1\rvert\lambda + \lvert k_2\rvert\sqrt{1-\lambda^2} \ge \lambda + \sqrt{1-\lambda^2} \ge 1$,
\end{center}
and
\begin{center}
    $\left \vert \dfrac{1}{\lvert k_1\rvert\lambda +\lvert k_2\rvert\sqrt{1-\lambda^2}} \right \rvert \le 1$.
\end{center}
We now have
\begin{center}
$\displaystyle \left\lvert \dfrac{1}{\lvert k_1\rvert\lambda -\lvert k_2\rvert\sqrt{1-\lambda^2}} \right\rvert \cdot \displaystyle \left\lvert \dfrac{1}{\lvert k_1\rvert\lambda +\lvert k_2\rvert\sqrt{1-\lambda^2}} \right\rvert < \dfrac{1}{C}( \lvert k_1\rvert+\lvert k_2\rvert) ^{1+\beta} \cdot 1 = \dfrac{1}{C}( \lvert k_1\rvert+\lvert k_2\rvert) ^{1+\beta}$,
\end{center}
and therefore,
\begin{center}
    $\displaystyle \left\lvert \dfrac{1}{-k_2^2+k_1^2\lambda^2+k_2^2\lambda^2} \right\rvert < \dfrac{1}{C}(\lvert k_1\rvert+\lvert k_2\rvert) ^{1+\beta}$.
\end{center}

We know that $( \lvert k_1 \rvert + \lvert k_2 \rvert )$ is equivalent to the $L^1$-norm, and that $\sqrt{k_1^2+k_2^2}$ is the $L^2$ norm, so we have a $C > 0$ such that
\begin{center}
    $\dfrac{1}{C}( \lvert k_1 \rvert + \lvert k_2 \rvert ) \le \sqrt{k_1^2+k_2^2} \le C ( \lvert k_1 \rvert + \lvert k_2 \rvert )$.
\end{center}

Hence, we have the inequality
\begin{center}
    $\displaystyle \left\lvert \dfrac{1}{-k_2^2+k_1^2\lambda^2+k_2^2\lambda^2} \right\rvert < \dfrac{1}{C}(\lvert k_1\rvert+\lvert k_2\rvert) ^{1+\beta} \le C\left (\sqrt{k_1^2+k_2^2}\right )^{1+\beta} < C\left(\sqrt{1+k_1^2+k_2^2}\right)^{1+\beta}$,
\end{center}
as desired.
\end{proof}

We use this to complete the proof of Proposition \ref{prop:bigboundingprop}.

\begin{proof}[Proof of Proposition \ref{prop:bigboundingprop}]
From Equation (\ref{eq:fourierexplicit}), we know 
\begin{equation*}
    \hat{u}(t,k_1,k_2) = \dfrac{\hat{f}(x) \cos \lambda t}{4 \pi^2 (-k_2^2 + k_1^2 \lambda^2 + k_2^2 \lambda^2)},
\end{equation*}
and we can write
\begin{align*}
    \left \vert \dfrac{\hat{f}(x) \cos \lambda t}{4 \pi^2 (-k_2^2 + k_1^2 \lambda^2 + k_2^2 \lambda^2)} \right \rvert &\le \left \vert \dfrac{\hat{f}(x)}{4 \pi^2 (-k_2^2 + k_1^2 \lambda^2 + k_2^2 \lambda^2)} \right \rvert \\ &= \left \vert \hat{f}(x) \cdot \dfrac{1}{4\pi^2} \cdot \dfrac{1}{-k_2^2 + k_1^2 \lambda^2 + k_2^2 \lambda^2} \right \rvert \\ &< \lvert \hat{f}(x) \rvert \cdot C\left(\sqrt{1+k_1^2+k_2^2}\right)^{1+\beta}.
\end{align*}
\end{proof}

\subsection{Differentiability of $u$} \label{regularity}

Now we use these bounds on $\hat{u}$ and substitute them into the Sobolev space definitions (Definition \ref{def:sobolev}). 

\begin{lemma}\label{lemma:ftouh}
If $f \in H^s$, then $u \in H^{s-1-\beta}$.
\end{lemma}
\begin{proof}
Because $f \in H^s$, 
\begin{equation*}
    \displaystyle \sum_{k_1,k_2=-\infty}^\infty (1+k_1^2+k_2^2)^{2s} \lvert \hat{f}(k_1,k_2) \rvert^2 < \infty. 
\end{equation*}
Using Proposition \ref{prop:bigboundingprop}, we have

$\displaystyle \sum_{k_1,k_2=-\infty}^\infty (1+k_1^2+k_2^2)^{2(s-1-\beta)} \lvert \hat{u}(k_1,k_2) \rvert^2$

\hspace{3cm} $< \displaystyle \sum_{k_1,k_2=-\infty}^\infty (1+k_1^2+k_2^2)^2({s-1-\beta}) \cdot \lvert (1+k_1^2+k_2^2)^{2(1+\beta)} \cdot \hat{f}(k_1,k_2) \rvert^2$

\hspace{3cm} $= \displaystyle \sum_{k_1,k_2=-\infty}^\infty (1+k_1^2+k_2^2)^{2s} \lvert \hat{f}(k_1,k_2) \rvert^2  < \infty$.

Hence, we know
\begin{equation*}
    \displaystyle \sum_{k_1,k_2=-\infty}^\infty (1+k_1^2+k_2^2)^{2(s-1-\beta)} \lvert \hat{u}(k_1,k_2) \rvert^2 < \infty,
\end{equation*}
meaning that $u \in H^{s-1-\beta}$. 

\end{proof}

Now we prove Theorem \ref{theorem:bigtheorem}, our main result. 
\begin{proof}[Proof of Theorem \ref{theorem:bigtheorem}]
By Lemma \ref{lemma:ftouh}, we have $u \in H^{s-1-\beta}$ when $f \in H^s$. Hence, by Proposition \ref{prop:htoc}, we know that $u \in C^{s-2-\beta}$ and $f \in C^{s-1}$; or, equivalently, when $f \in C^s$, then $u \in C^{s-1-\beta}$, as desired.
\end{proof}

This also allows us to show that $u$ is smooth given $f$ is smooth. 

\begin{corollary} \label{corollary:smoothcorollary}
If $f$ is smooth (i.e. $f \in C^\infty$), then $u$ is also smooth. 
\end{corollary}
\begin{proof}
If $f$ is smooth, then $f \in H^s$ for all $s$. Then, $u \in H^{s-1-\beta}$ for all $s$, which implies that $u \in C^{s-2-\beta}$ for all $s$; so, $u$ is also smooth. 
\end{proof}
\vspace{-0.1cm}
In this section, we directly related the regularity of $f$ to the regularity of $u$; more precisely, that when $f \in C^{s}$, then $u \in C^{s-1-\beta}$. We used Corollary \ref{corollary:smoothcorollary} to extend this to smoothness in general, concluding our main result regarding the regularity of solutions to the Poincaré problem in the square for sufficiently irrational rotation. 
\vspace{-2cm}
\section{Numerical Simulations on \textit{r} and Future Directions} \label{section4}

The natural extension of these results is to examine the behavior of $u$ and irrational rotation in shapes beyond the square. We conjecture that we can extend the smoothness result to other shapes.

\begin{conjecture}
Given a smooth forcing function $f(x)$ and a sufficiently irrational rotation number $r(\lambda)$, the solution $u(t)$ of the Poincaré problem, acting on any convex $\lambda$-simple domain (see Dyatlov et al. \cite{DWZ}, Definition 1), is smooth.
\end{conjecture}

Studying the chess billiard flow is more difficult in other shapes where rotation numbers cannot be explicitly calculated. To study them, we can use numerical simulations of $r$. We calculate the value of $r(\lambda)$ by iteratively intersecting the boundary of the shape with lines of slope $\pm\dfrac{\sqrt{1-\lambda^2}}{\lambda}$, to map the points $b^k(p)$. To estimate $r(\lambda)$ and distance traveled per rotation, we parameterize using angles. These calculations were performed on Wolfram Mathematica, Version 12.1.1.0. 

Figure \ref{fig:square_simulated_r} shows the numerical simulations plot of $r(\lambda)$ vs. $\lambda$ for a square, which we see aligns well with the plot of $\lambda$ vs.  $\dfrac{\lambda}{\sqrt{1-\lambda^2}+\lambda}$, shown in Figure \ref{fig:square_real_r}. These plots are smooth, which is not necessarily true for other geometries, as shown in the lower subplots of Figure \ref{fig:shapes_rot}. 

\begin{figure}[H]
    \centering
    \begin{subfigure}[b]{0.42\textwidth}
        \centering
         \includegraphics[width=\textwidth]{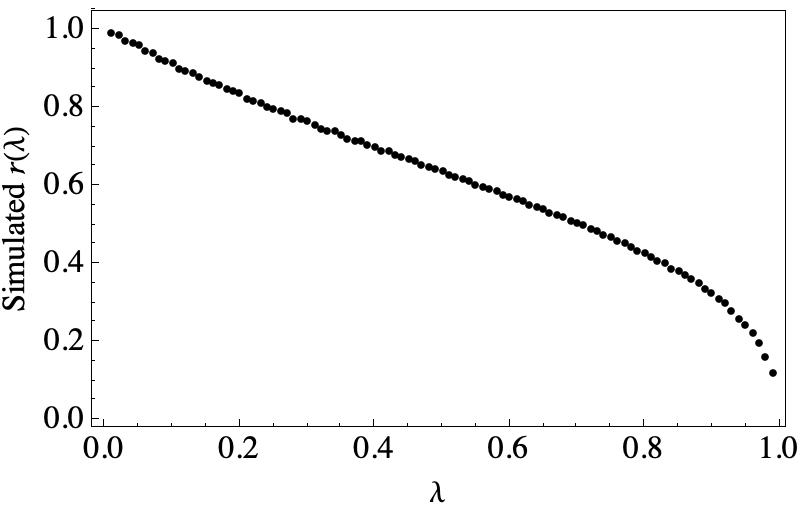}
         \caption{Simulated $r(\lambda)$}
         \label{fig:square_simulated_r}
    \end{subfigure}
    \begin{subfigure}[b]{0.42\textwidth}
        \centering
         \includegraphics[width=\textwidth]{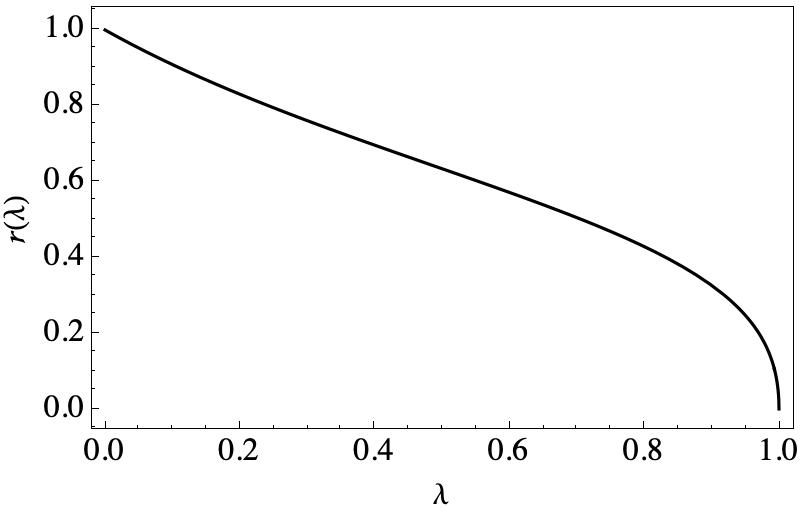}
         \caption{$r(\lambda) = \frac{\lambda}{\sqrt{1-\lambda^2}+\lambda}$}
         \label{fig:square_real_r}
    \end{subfigure}
    \caption{Simulated and expected plots of $r(\lambda)$ vs. $\lambda$ in a square.}
    \label{fig:square_rot}
\end{figure}

For the rotation number plots in Figure \ref{fig:shapes_rot}, we sampled $r(\lambda)$ for 10,000 equally-spaced values of $\lambda$. Here, we see there are plateaus, or flat segments, that form. Even with small adjustments from the square, such as a small angle of perturbation for the tilted square, plateaus form and the plots are no longer smooth. 

\begin{figure}[H]
    \centering
    \begin{subfigure}[b]{0.32\textwidth}
        \centering
         \includegraphics[width=\textwidth]{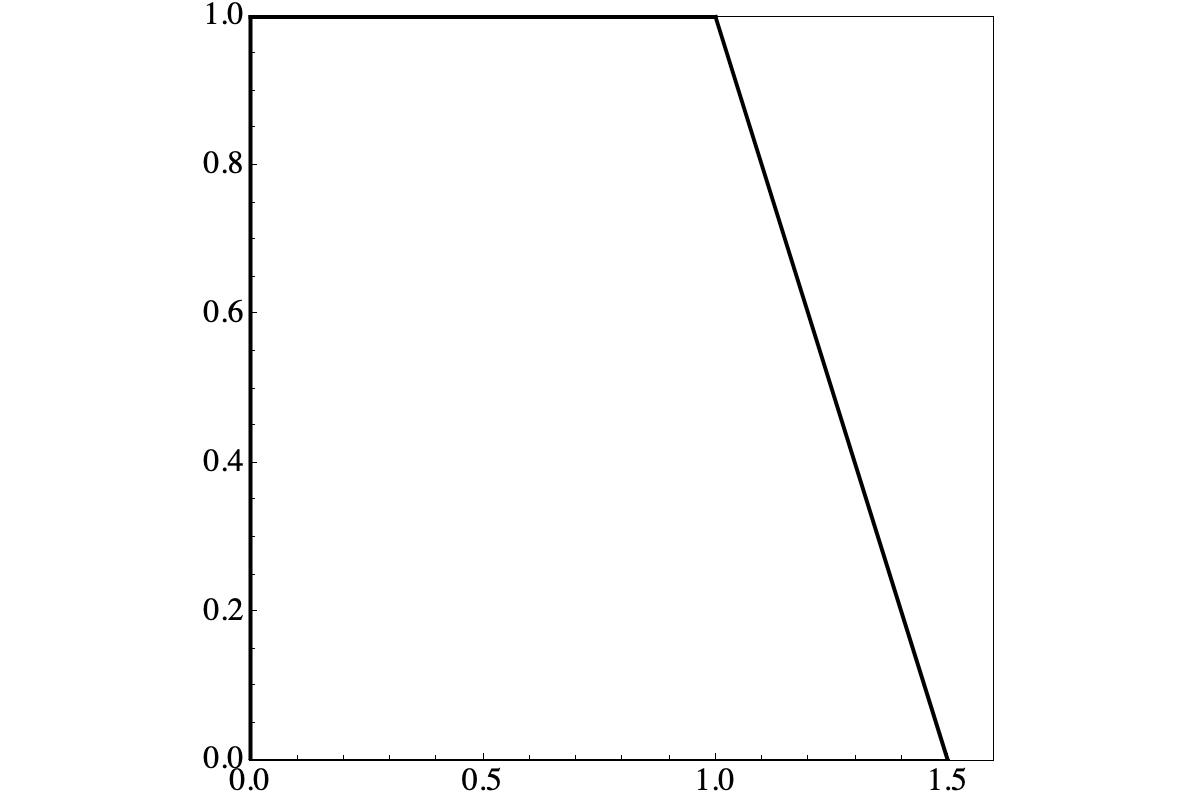}
    \end{subfigure}
    \hfill
    \begin{subfigure}[b]{0.32\textwidth}
        \centering
         \includegraphics[width=\textwidth]{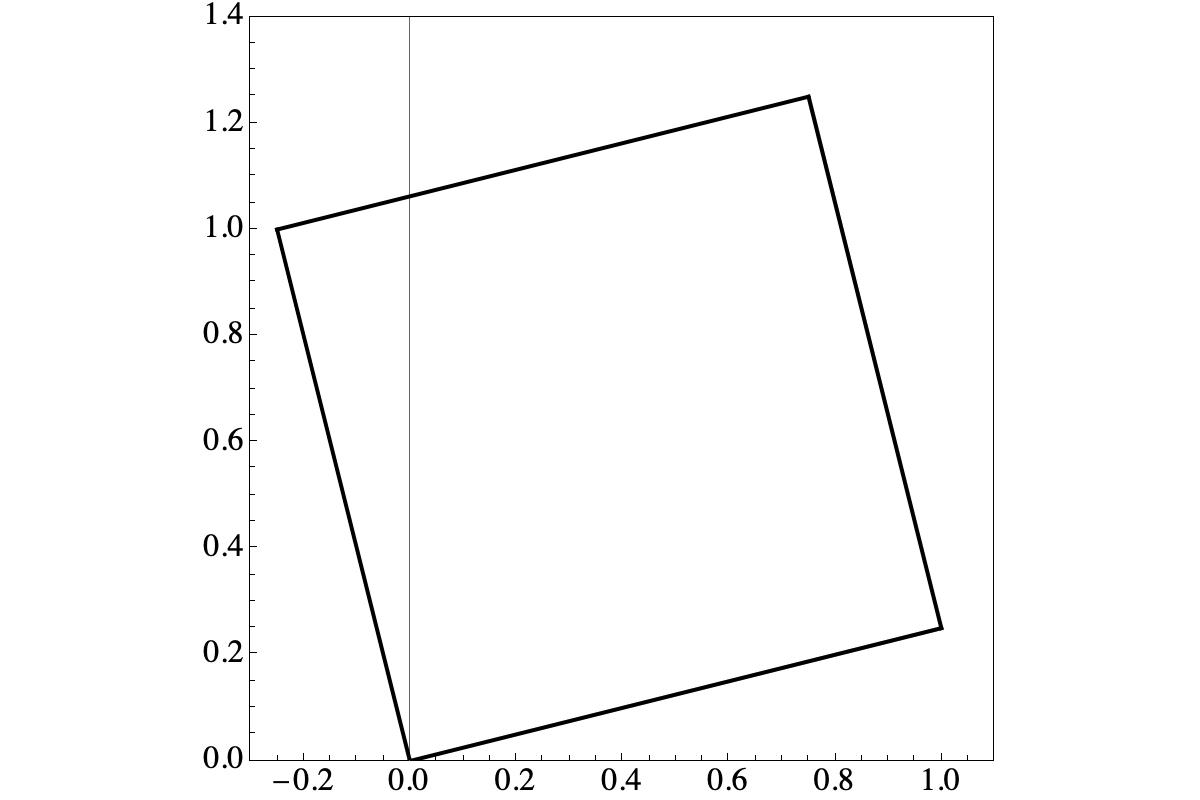}
    \end{subfigure}
    \hfill
    \begin{subfigure}[b]{0.32\textwidth}
        \centering
         \includegraphics[width=\textwidth]{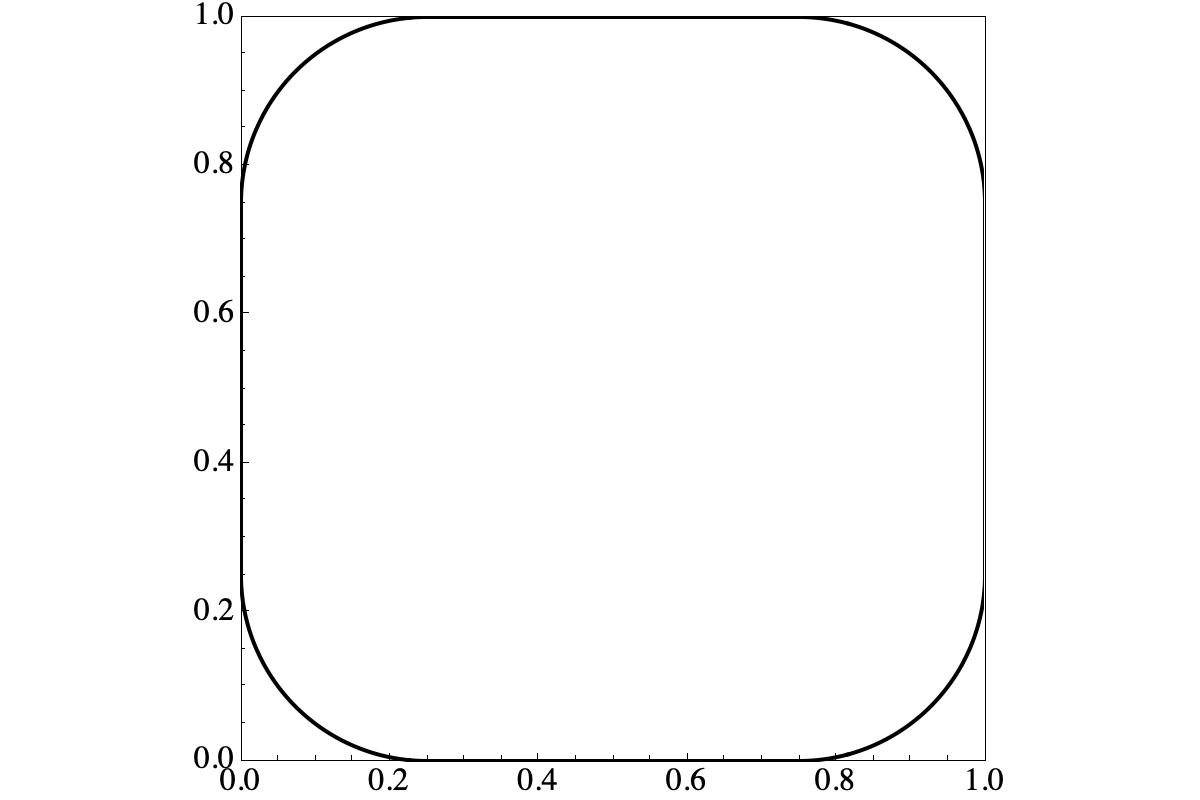}
    \end{subfigure}
    \hfill
    \vspace{0.5cm}
    
    \begin{subfigure}[b]{0.32\textwidth}
        \centering
         \includegraphics[width=\textwidth]{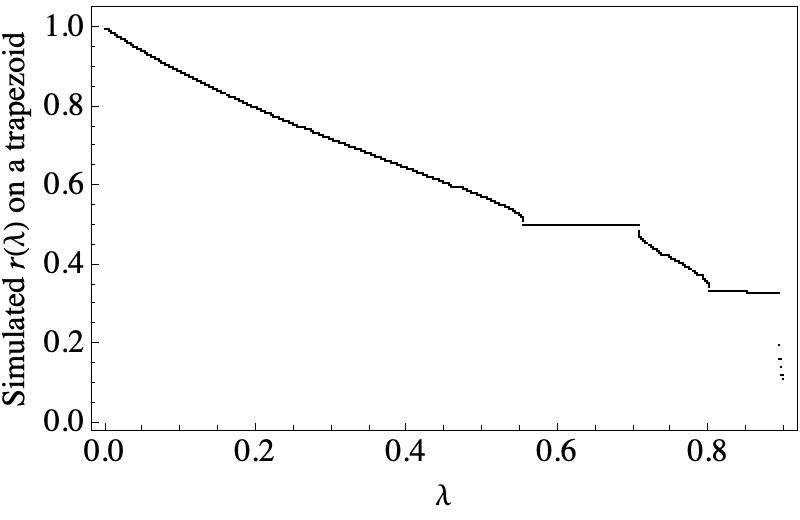}
         \caption{$r(\lambda)$ plot for a trapezoid}
    \end{subfigure}
    \hfill
    \begin{subfigure}[b]{0.32\textwidth}
        \centering
         \includegraphics[width=\textwidth]{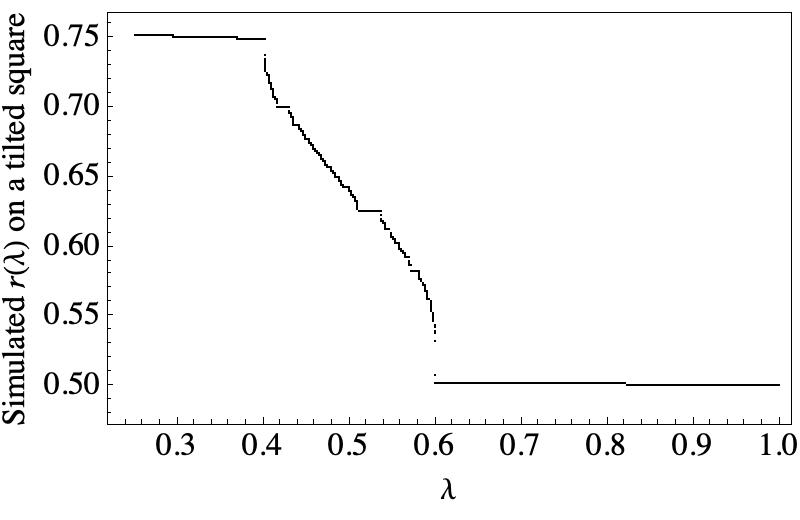}
         \caption{$r(\lambda)$ plot for a tilted square}
    \end{subfigure}
    \hfill
    \begin{subfigure}[b]{0.32\textwidth}
        \centering
         \includegraphics[width=\textwidth]{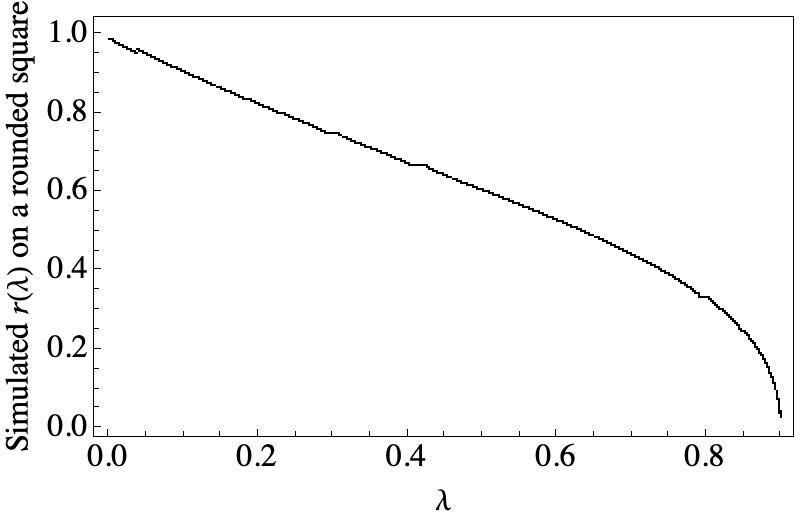}
         \caption{$r(\lambda)$ plot for a rounded square}
    \end{subfigure}
    
    \caption{Numerical simulations for $r(\lambda)$ on various shapes, with visible plateaus. }
    \label{fig:shapes_rot}
\end{figure}

One way to examine the formation of plateaus is through fractal similarity and dimension for $r(\lambda)$, described in the next subsection. 

\subsection{Devil's Staircase Dimensional Analysis}

The Devil's Staircase dimension is a measure of fractal dimension for plots with plateaus, and can be used to study dynamical systems. Jensen et al. \cite{jensen} describe this value as calculated by approximating the Minkowski dimension on the set where there are no plateaus. The Minkowski dimension estimates a curve's fractal similarity by counting the number of $\epsilon$ size tiles needed to cover points on the curve completely. More precisely, we define $S$ as the fraction of points on a plateau, and $1-S$ as the fraction of points between plateaus, and $q(\epsilon)$ as the number of tiles of size $\epsilon$ encountered by the points between plateaus.

The length $N(\epsilon)$ covered by the tiles is
\begin{equation*}
    N(\epsilon) = \dfrac{1-S}{q(\epsilon)},
\end{equation*}
and the Devil's Staircase dimension is 
\begin{equation*}
    D(\epsilon) = \dfrac{\log{N(\epsilon)}}{\log{1/q(\epsilon)}}.
\end{equation*}

We see that when a plot has no plateaus, then there are no gaps in the set and $N=1/q$, and the dimension is 1. For plots with increasingly large and prevalent plateaus, we expect the dimension to decrease to 0. 

We analyze the dimension of the tilted square, for varying angles of tilt. We know that when the angle of tilt is 0, then it is a normal square, and the dimension should be 1. The plot in Figure \ref{fig:dimensionplot} shows the change in dimension as the angle of tilt increases. We fit a polynomial curve to this plot, and we find that the curve can be fit well to a quadratic. 
\begin{figure}[H]
    \centering
    \begin{subfigure}[b]{0.42\textwidth}
        \centering
         \includegraphics[width=\textwidth]{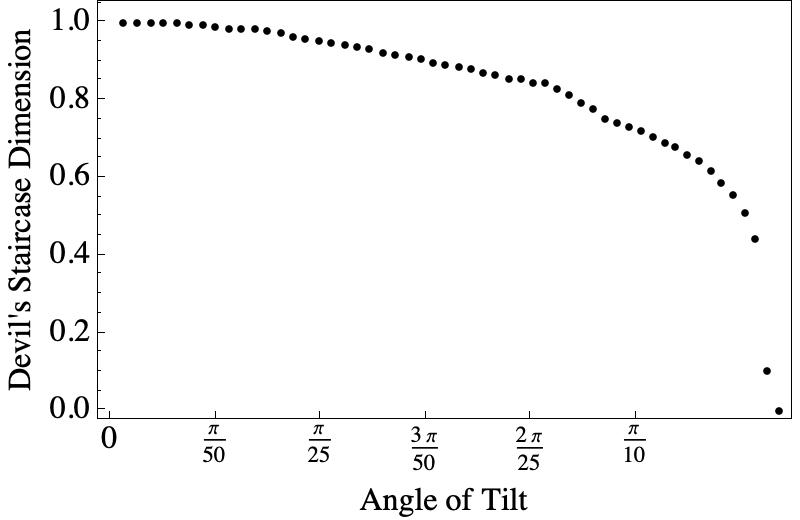}
    \end{subfigure}
    \begin{subfigure}[b]{0.42\textwidth}
        \centering
         \includegraphics[width=\textwidth]{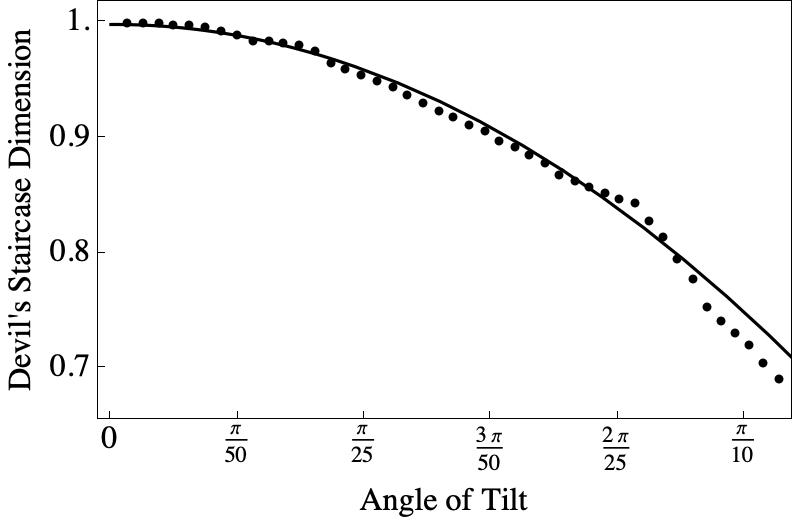}
    \end{subfigure}
    \caption{Devil's Staircase dimension for $r(\lambda)$ at varying angles of tilt on a tilted square.}
    \label{fig:dimensionplot}
\end{figure}

\section{Conclusion}

We showed that the solution $u$ to the Poincaré problem is smooth in the square for sufficiently irrational rotation, extending past results regarding rational rotation. We explicitly found the rotation number of the chess billiard map in the square, and used its irrationality condition to bound the Fourier coefficients of $u$, ultimately concluding that $u$ retains almost all of the regularity of $f$, and that high regularity and smoothness of $f$ correspond to high regularity and smoothness of $u$. 

Further study in this topic would include extending this result to more or all geometries; specifically, analyzing rapid decay of Fourier coefficients in irrational rotation for shapes beyond a square. We expect that in other boundaries with sufficiently irrational rotation, the Poincaré problem will also result in smooth solutions. For most other geometries, $r(\lambda)$ is difficult to formulate explicitly, and numerical simulations would play a useful role in future work in this direction.

\section{Acknowledgments} 

Thank you to my mentor Zhenhao Li at the Massachusetts Institute of Technology (MIT) Department of Mathematics for suggesting this project and his constant support and guidance through this mentorship. Thank you to the head mentor, Dr. Tanya Khovanova from the math department of MIT for her advice and discussions during our meetings. I would also like to thank my tutor Dr. John Rickert of the Rose-Hulman Institute of Technology for providing rounds of feedback on my paper and presentation and for answering my many questions. Thank you as well to Dr. David Jerison and Dr. Ankur Moitra from MIT for arranging and supervising this mentorship. I also thank Yunseo Choi for her help on redrafting my final paper, as well as past RSI students Peter Gaydarov, Kevin Cong, Victor Kolev, Miles Dillon Edwards, and Donald Liveoak for their feedback on my paper. In addition, I thank the Massachusetts Institute of Technology and the Center for Excellence in Education for hosting the Research Science Institute and arranging such an amazing summer mentorship experience for me. Finally, thank you to my sponsors for giving me the opportunity to attend this program; your support and generosity have given me an exceptional learning experience: Mr. George Thagard, Director of the Thagard Foundation; Ms. Robin Wright, Trustee of the Arnold and Kay Clejan Charitable Foundation; Mr. Victor Liu, President of the Building Futures Foundation; Dr. Jin Zhang and Dr. Eileen H. You, Mr. Joseph Lui and Ms. Madelina Ma, Mr. Sam Sivakumar and Ms. Rajshree Sankaran, Mr. and Mrs. Shane R. Albright, and Mr. Arvind Parthasarathi. 